\documentclass[12pt]{amsart}
 \usepackage{amsmath,amsfonts,amssymb}
 \usepackage{amsthm}
 \usepackage{array}
 \usepackage{graphics}
 \usepackage{amsmath}
 \usepackage{amssymb}
 \usepackage{amsmath,amsthm,mathrsfs,amsfonts,
amssymb,times,latexsym}

 \usepackage{t1enc}
 \usepackage[utf8]{inputenc}

 \newtheorem{thm}{Theorem}
 \newtheorem{lem}{Lemma}

 \newtheorem{rem}{Remark}

\newcommand{\z}{\mathbb{Z}}

\newcommand{\al}{\alpha}
\newcommand{\kk}{\mathbb{K}}

\newcommand{\R}{\mathbb{R}}
\newcommand{\C}{\mathbb{C}}

\makeatletter
\@namedef{subjclassname@2020}{%
  \textup{2020} Mathematics Subject Classification}
\makeatother

\begin{document}

\title{On the largest value of the solutions of Erd\H{o}s's last equation}

\author{Istv\'an~Pink}
\address{Institute of Mathematics, University of Debrecen,\newline
\indent P. O. Box 400, H-4002 Debrecen, Hungary}
\email{pinki@science.unideb.hu}

\author{Csaba~S\'andor}
\address{Department of Stochastics, Institute of Mathemetics, Budapest University of Technology and Economics, M\H{u}egyetem rkp. 3., H-1111 Budapest, Hungary; \newline \hspace*{4mm}
HUN-REN Alfr\'ed R\'enyi Institute of Mathematics, Re\'altanoda utca 13--15., H-1053 Budapest, Hungary
MTA--HUN-REN RI Lend\"ulet ``Momentum'' Arithmetic Combinatorics Research Group, Re\'altanoda utca 13--15., H-1053 Budapest, Hungary}
\email{sandor.csaba@ttk.bme.hu}

\thanks{The first author is supported by the NKFIH grant ANN 130909.
The second author is supported by the NKFIH Grants No. K129335, the
grant NKFI KKP144059 ''Fractal geometry and applications'' and the Lend\"ulet ''Momentum'' program of the Hungarian Academy of Sciences (MTA)}
\keywords{Erd\H{o}s's last equation, Baker's method, LLL algorithm}
\subjclass[2020]{Primary 11D45, Secondary 11D72}
\date{\today}

\maketitle

\begin{abstract}
Let $n$ be a positive integer. The Diophantine equation $n(x_1+x_2+\dots +x_n)=x_1x_2\dots x_n$, $1 \le x_1\le x_2\le \dots \le x_n$ is called Erd\H{o}s's last equation. We prove that $x_n\to \infty $ as $n\to \infty$ and determine all tuples $(n,x_1,\dots ,x_n)$ with $x_n\le 10$.
\end{abstract}

\section{Introduction}

On 19 August 1996, a month before his death, Paul Erd\H{o}s formulated the following problem in a letter written to Richard Guy (see \cite{Guy}):

"Do you know anything about the equation
\begin{equation}\label{ELE}
n(x_1+x_2+\dots +x_n)=x_1x_2\dots x_n
\end{equation}
$1 \le x_1\le x_2\le \dots \le x_n$, where the $x_i$ are positive integers?
Denote the number of solution in $n$ by $f(n)$; $f(n)>n^{\varepsilon }$ for some $\varepsilon >0$ seems to be true? ...
"

\medskip
The equation (\ref{ELE}) is called Erd\H{o}s last equation. It is easy to see that for $x_1=\dots =x_{n-2}=1$, equation (\ref{ELE}) is equivalent to $n(2n-2)=(x_{n-1}-n)(x_n-n)$, therefore $f(n)\ge \lceil \frac{1}{2}d(2n^2-2n) \rceil >0$, where $d(.)$ is the number of divisors function.

\medskip

A solution of equation (\ref{ELE}) is
$$
x_1=\dots =x_{n-2}=1, \ x_{n-1}=n+1, \ x_n=2n^2-n.
$$
Shiu \cite{Shiu} proved that in fact, for any solution of (\ref{ELE}) we have $x_n\le 2n^2-n$.
Denote by $g(n)$ the least possible value of $x_n$ for equation (\ref{ELE}).
The solution
$$
x_1=\dots =x_{n-2}=1, \ x_{n-1}=2n, \ x_n=3n-2
$$
implies that $g(n)\le 3n-2$.
For $n=m^2$, Shiu \cite{Shiu} gave the solution
$$
x_1=\dots x_{m^2-4}=1, \ x_{m^2-3}=x_{m^2-2}=x_{m^2-1}=m, \ x_{m^2}=m+4,
$$
which yields
\begin{equation} \label{shiu-result}
\displaystyle\liminf_{n\to \infty}\frac{g(n)}{\sqrt{n}}\le 1.
\end{equation}

\medskip

\noindent The aim of this paper is twofold. In one hand, we study the function $g(n)$ (the possible minimum value of $x_n$
in \eqref{ELE}) and derive two results concerning it (see Theorems \ref{thm1} and \ref{thm3}).
On the other hand, we deduce an effective finiteness result on equation \eqref{ELE} (see Theorem \ref{thm2}),
and by combining this with the LLL algorithm, we solve completely equation \eqref{ELE} for $3 \le x_n \le 10$
(see Theorem \ref{thm4}).

\medskip

\noindent Our first theorem 
improves the upper bound given in \eqref{shiu-result}.

\begin{thm}\label{thm1} Consider equation \eqref{ELE}. Then, we have $\displaystyle \liminf_{n\to \infty }\frac{g(n)}{\sqrt[3]{n}}\le 1$.


\end{thm}

\medskip

Observe that $n=1=x_n$ is always a solution to \eqref{ELE}. We will call such a solution {\it trivial}.
In what follows, we concentrate to the case where in equation \eqref{ELE} the term $x_n$ is fixed.

If $x_n=1$, then since $1 \le x_1 \le x_2 \le \ldots \le x_n$, we have $x_1=\ldots=x_n=1$ and
equation \eqref{ELE} reduces to $n^2=1$, yielding $n=1$. This shows that for $x_n=1$ the only solution to
\eqref{ELE} is the trivial solution $n=1=x_n$.


If $x_n=2$, then $x_1x_2\dots x_n=2^m$. Hence $n=2^u$ and $2^u<x_1+\dots +x_n\le 2\cdot 2^u$, but $x_1+x_2+\dots +x_n=2^v$ implies that $2^v=2^{u+1}$, that is $x_i=2$ for every $1\le i\le n$. Hence $2n^2=2^n$, which is impossible for $n \ge 2$.
It follows that for $x_n = 2$ the only solution to the equation (\ref{ELE}) is the trivial one.

In what follows, we assume that $x_n \ge 3$ is fixed and for $1,2,\ldots, x_n$ let $a_1, a_2, \ldots, a_{x_n}$ denote the number of ones, twos, $\ldots, x_n$s among the numbers $1 \le x_1 \le x_2 \ldots \le x_n$. Then, we clearly have that
\begin{equation} \label{eq:prelim-1}
a_1+\ldots +a_{x_n}=n \ \textrm{with} \ a_1 \ge 0,\ldots,a_{x_n-1}\ge 0, a_{x_n} \ge 1
\end{equation}
and
\begin{equation}\label{eq:prelim-2}
x_1+\ldots+x_n=a_1+2a_2+\ldots+x_na_{x_n}.
\end{equation}
Thus, the combination of \eqref{eq:prelim-1} and \eqref{eq:prelim-2} shows that equation \eqref{ELE} can be rewritten as


\begin{equation}\label{eq:main}
n(n+a_2+2a_3+\ldots+(x_n-1)a_{x_n})=2^{a_2}3^{a_3}\cdots x_n^{a_{x_n}},
\end{equation}
where the nonnegative integer unknowns $(n,a_2,a_3,\ldots,a_{x_n})$ satisfy
\begin{equation} \label{eq:main-ass}
n-(a_2+a_3+\ldots+a_{x_n})=a_1 \ge 0 \ \textrm{with} \ n \ge 1 \ \textrm{and} \ x_n \ge 3.
\end{equation}

Let $k:=\pi(x_n)$, (where $\pi(x)$ counts the number of prime numbers less than or equal to some real number $x$)
and put $2=p_1<p_2<\ldots<p_k=p_{\pi(x_n)}$. The product $2^{a_2}3^{a_3}\cdots x_n^{a_{x_n}}$ on the right hand side of
\eqref{eq:main} can be rewritten as $2^{a_2}3^{a_3}\cdots x_n^{a_{x_n}}=p_1^{b_1}\cdots p_k^{b_k}$,
where
\begin{equation} \label{eq:thm2-1}
b_j:=\sum_{i=1}^{\lfloor x_n/p_j \rfloor}{a_{p_j  i}}\cdot {\nu_{p_j}(a_{p_j i})}, \quad (1 \le j \le k),
\end{equation}
where, for a prime $p$ and a nonzero integer $m$, $\nu_p(m)$ denotes the largest nonnegative integer $u$ such that $p^u \mid m$.
By putting $b:=b_1+\ldots+b_k$, relation \eqref{eq:thm2-1} implies that
\begin{equation}\label{eq:thm2-2}
b \ge a_2+\ldots+a_{x_n}.
\end{equation}
Since $2^{a_2}3^{a_3}\cdots x_n^{a_{x_n}}=p_1^{b_1}\cdots p_k^{b_k}$, equation \eqref{eq:main} reduces to
\begin{equation}\label{eq:main2}
n(n+a_2+2a_3+\ldots+(x_n-1)a_{x_n})=p_1^{b_1}\cdots p_k^{b_k}.
\end{equation}

First, we derive an explicit upper bound for $b=b_1+\cdots+b_k$ depending only on $x_n$ in equation \eqref{eq:main2}.

\begin{thm} \label{thm2}
Consider equation \eqref{eq:main2} in integer unknowns $(n,a_2,a_3,\ldots,a_{x_n})$ satisfying
\eqref{eq:main-ass} and \eqref{eq:thm2-1}. Then for $b=b_1+\cdots+b_k$ one has
$$
b \le B_0,
$$
where $B_0$ is any upper bound for the largest integer solution of the inequality
\begin{equation} \label{eq:thm2-bound}
b<(2.8\cdot30^{k+3}\cdot k^{4.5} )\cdot (\log{p_2})\cdots(\log{p_k})(1+\log{b})+\frac{2\log{((x_n-1) b)}}{\log{p_1}}.
\end{equation}
\end{thm}

\vskip.2cm\noindent
\begin{rem} \label{remark1}
{\rm
The combination of Theorem \ref{thm2} and \eqref{eq:thm2-2} shows that
\begin{equation} \label{rem1}
\max\{a_2,\ldots,a_{x_n}\} \le a_2+\ldots+a_{x_n} \le b \le B_0.
\end{equation}
Thus, \eqref{eq:prelim-1}, \eqref{eq:main2} and  \eqref{rem1} give
\begin{equation} \label{rem2}
a_1<n \le \sqrt{p_1^{b_1}\cdots p_k^{b_k}}=\sqrt{2^{a_2}\cdots x_n^{a_{x_n}}} \le x_n^{\frac{a_2+\ldots+a_{x_n}}{2}} \le x_n^{\frac{B_0}{2}}.
\end{equation}
Hence, \eqref{rem1} and \eqref{rem2} show that for given $x_n \ge 3$, Theorem \ref{thm2} provides an effective upper bound depending only on $x_n$ for all possible solutions $(n,a_2,a_3,\ldots,a_n)$ of equations \eqref{eq:main2} and \eqref{eq:main}. Thus, for given $x_n \ge 3$ all solutions $(n,x_1,\ldots,x_n)$ of equation \eqref{ELE} are bounded, as well.
}
\end{rem}

The next theorem provides a quantitative lower bound for $g(n)$.

\begin{thm}\label{thm3} Consider equation \eqref{ELE}. Then, we have $$\displaystyle \liminf_{n\to \infty}\frac{g(n)}{\frac{\log \log n \cdot \log \log \log n}{\log \log \log \log n}}\ge 1.$$
\end{thm}

\begin{rem} \label{remark2}
{\rm
We note that Theorem \ref{thm3} implies that $x_n\to \infty $ as $n\to \infty$ in a quantitative form.
}
\end{rem}

\vskip.2cm \noindent
By combining Theorem \ref{thm2} with the LLL algorithm, we solve completely equation \eqref{eq:main} for $3 \le x_n \le 10$.
It is clear that from the solutions of \eqref{eq:main} we can easily derive all the solutions of equation \eqref{ELE} with $3 \le x_n \le 10$.

\begin{thm} \label{thm4}
Consider equation \eqref{eq:main} satisfying \eqref{eq:main-ass}.
\begin{itemize}
  \item There are exactly $5$ solutions with $x_n=3$, namely
  {\small
  $$
  [n,a_1,a_2,a_3] \in \{[ 1, 0, 0, 1 ],[ 3, 0, 0, 3 ],[ 9, 4, 1, 4 ],[ 36, 27, 6, 3 ],[ 81, 71, 5, 5 ]\}.
  $$
  }
  \item There are exactly $11$ solutions with $x_n=4$, namely
  {\small
  $$
  [n,a_1,a_2,a_3,a_4] \in \{[ 1, 0, 0, 0, 1 ],
    [ 2, 0, 0, 0, 2 ],
    [ 4, 1, 0, 1, 2 ],
    [ 8, 2, 5, 0, 1 ],
    [ 8, 4, 0, 2, 2 ],
    [ 16, 10, 2, 3, 1 ],
$$
$$
    [ 24, 17, 3, 3, 1 ],
    [ 48, 41, 2, 1, 4 ],
    [ 64, 57, 0, 4, 3 ],
    [ 128, 116, 9, 2, 1 ],
    [ 144, 134, 3, 6, 1 ]\}.
  $$
}
  \item There are exactly $34$ solutions with $x_n=5$ and all of them satisfy $n \le 32768$.
  \item There are exactly $57$ solutions with $x_n=6$ and all of them satisfy $n \le 32768$.
  \item There are exactly $160$ solutions with $x_n=7$ and all of them satisfy $n \le 279936$.
  \item There are exactly $172$ solutions with $x_n=8$ and all of them satisfy $n \le 279936$.
  \item There are exactly $330$ solutions with $x_n=9$ and all of them satisfy $n \le 354294$.
  \item There are exactly $613$ solutions with $x_n=10$ and all of them satisfy $n \le 367416$.
\end{itemize}
\end{thm}

\noindent We note, that in the cases $5 \le x_n \le 10$, we refrain from listing explicitly
all the solutions because of the large number of them. All solutions are available on request from the author.



\section{Auxiliary results}

Let $\eta$ be an algebraic number of degree $d$ whose
minimal polynomial over the integers is $$f(x) = a_0 \prod_{i=1}^d (x - \eta_i)$$with $a_0\ge 1$.
The {\it absolute logarithmic height} of $\eta$ is defined as
$$
h(\eta):= \frac{1}{d}\left( \log a_0 + \sum_{i=1}^d \log \max \{|\eta_i|,1\}\right).
$$

Let $\kk \subset \mathbb{R}$ be an algebraic number field and let $D$ be the degree of
the field $\kk$. Let $\eta_1, \eta_2, \ldots, \eta_l$ be nonzero elements of $\kk$ and $b_1, \ldots, b_l$ be integers.
We put
\begin{equation}\label{eq:Lambda}
\Lambda = \prod_{i=1}^l \eta_i^{b_i} -1,
\end{equation}
and
\begin{equation} \label{eq:B}
B  \ge \max\{|b_1|, \ldots, |b_l|\}.
\end{equation}
Let $A_1, \ldots, A_l$ be positive integers such that $$A_j \geq h'(\eta_j) := \max \{Dh(\eta_j), |\log \eta_j|, 0.16
\}\quad {\text{\rm for}}\quad j=1,\ldots l.
$$

The  following is contained in Theorem 9.4 of \cite{BMS} which is a consequence of a result due to Matveev \cite{Matv}.

\begin{lem} \label{lem:Matveev}
If $\Lambda \neq 0$ then
\begin{equation*}
\label{ineq:matveev} \log |\Lambda| > -C(l,D)(1+\log{B})A_1A_2\cdots A_l,
\end{equation*}
where
$$
C(l,D):=1.4 \cdot 30^{l+3}l^{4.5}D^2(1+ \log D).
$$
\end{lem}

\noindent Following Lenstra, Lenstra and Lov\'asz \cite{LLL}, we recall the definition of an {\it LLL-reduced basis} of a lattice $\mathcal{L}\subset \R^n$.
For a basis $\{b_1,b_2,\ldots,b_n\}$ of the lattice $\mathcal{L}$ the Gram-Schmidt procedure provides an orthogonal basis  $\{b_1^{*}, b_2^{*},\ldots,b_n^{*}\}$ of $\mathcal{L}$ with respect to the inner product $\langle \cdot , \cdot \rangle$ of $\R^n$ given inductively by
$$
b_i^{*}=b_i-\sum_{j=1}^{i-1}{\mu_{ij}b_j^{*}} \quad (1 \le i \le n), \quad \mu_{i,j}=\frac{\langle b_i,b_j^{*}\rangle }{\langle b_j^{*},b_j^{*}\rangle} \quad (1 \le j<i \le n).
$$
We call a basis $\{b_1,b_2,\ldots,b_n\}$ for a lattice $\mathcal{L}$ {\it LLL-reduced} if
$$
\|\mu_{i,j}\| \le \frac{1}{2} \quad (1 \le j<i \le n)
$$
and
$$
\|b_i^{*}+\mu_{i, i-1}b_{i-1}^{*}\|^2 \ge \frac{3}{4}\|b_{i-1}^{*}\|^2  \quad  (1 <i \le n),
$$
where $\|.\|$ denotes the ordinary Euclidean length.

To reduce the initial upper bounds for the parameters, we shall also need the following three standard lemmas.

\begin{lem}\label{L_shortest_vector_by_LLL}
Let $b_1, \dots ,b_n$ be an LLL-reduced basis of a lattice $\mathcal{L}\subset \R^n$. Then $C_1:=\sqrt{||b_1||^2/2^{n-1}}$ is a lower bound for the length of the shortest vector of $\mathcal{L}$.
\end{lem}

\begin{proof}
This is a simplified version of Theorem 5.9 of \cite{Smart}.
\end{proof}

\begin{lem}\label{L_LLL_for_linearforms}
Let $\al_1, \dots, \al_m\in \R$ be real numbers and $x_1, \dots, x_m \in \z$ with $|x_i| \leq X_i$. Put $X_0:=\max X_i$, $S:=\sum_{i=1}^{m-1} X_i^2$, $T:=0.5+0.5\cdot \sum_{i=1}^m X_i$ and assume that
$$
\left| \sum_{i=1}^m x_i\al_i \right| \leq C_2\exp\left\{ -C_3H^q \right\}
$$
holds for some positive real constants $C_2, C_3, H$ and positive integer $q$. Let $C\ge X_0^m$ and let $\mathcal{L}$ denote the
lattice of $\R^m$ generated by the columns of the matrix
$$
A=
\begin{pmatrix}
1 & \dots & 0 & 0 \\
\ & \ddots & \ & \ \\
0 & \dots & 1 & 0 & \\
[C\al_1] & \dots & [C\al_{m-1}] & [C\al_m]
\end{pmatrix}
\in \z^{n\times n}.
$$
Let $C_1$ denote a lower bound on the length of the shortest non-zero vector of the lattice $\mathcal{L}$. If $C_1^2 > T^2+S$ then we have
either
$$
H\leq \sqrt[q]{\frac{1}{C_3} \left( \log(C \cdot C_2)-\log\left( \sqrt{C_1^2-S}-T \right) \right) },
$$
or
$$
x_1=x_2=\dots = x_{m}=0.
$$
\end{lem}

\begin{proof}
This is Lemma VI.I of \cite{Smart}.
\end{proof}

\begin{lem}\label{SmartB2}
Let $z \in \C$ with $|z-1|\leq \frac{1}{2}$. Then
$$
\frac{|\log z|}{2} < |z-1|.
$$
\end{lem}

\begin{proof}
This is an easy consequence of Lemma B.2 of \cite{Smart}.
\end{proof}

\section{Proof of Theorem \ref{thm1}}
\begin{proof} For $n=m^3-3m^2-4m$, $x_1=\dots =x_{m^3-3m^2-4m-6}=1$, $x_{m^3-3m^2-4m-5}=m-1$, $x_{m^3-3m^2-4m-4}=x_{m^3-3m^2-4m-3}=m$, $x_{m^3-3m^2-4m-2}=m+1$, $x_{m^3-3m^2-4m-1}=m+2$, $x_{m^3-3m^2-4m}=m+4$ is a solution of equation (\ref{ELE}).
\end{proof}

\section{Proof of Theorem \ref{thm2}}


\begin{proof}

By \eqref{eq:main2}, $x_n \ge 3$ and $a_{x_n} \ge 1$, we easily see that
\begin{equation} \label{eq:thm2-3}
n+a_2+2a_3+\ldots+(x_n-1)a_{x_n}>\sqrt{p_1^{b_1}\cdots p_k^{b_k}} \ge p_1^{\frac{b_1+\ldots+b_k}{2}}=p_1^{\frac{b}{2}}
\end{equation}
and
\begin{equation} \label{eq:thm2-4}
n+a_2+2a_3+\ldots+(x_n-1)a_{x_n}>\sqrt{p_1^{b_1}\cdots p_k^{b_k}} \ge p_j^{\frac{b_j}{2}} \qquad (2 \le j \le k).
\end{equation}

\noindent Further, using \eqref{eq:main2} again, we may write that
$$
n=p_1^{b_1'}\cdots p_k^{b_k'} \quad \textrm{and} \quad n+a_2+2a_3+\ldots+(x_n-1)a_{x_n}=p_1^{b_1''}\cdots p_k^{b_k''},
$$
with some nonnegative integers $b_i', \, b_i''$ with $b_i'+b_i''=b_i \, (1 \le i \le k)$.
Thus,
$$
p_1^{b_1'-b_1''}\cdots p_k^{b_k'-b_k"}=\frac{n}{n+a_2+2a_3+\ldots+(x_n-1)a_{x_n}}=1-\frac{a_2+2a_3+\ldots+(x_n-1)a_{x_n}}{n+a_2+2a_3+\ldots+(x_n-1)a_{x_n}},
$$
which by
$$
a_2+2a_3+\ldots+(x_n-1)a_{x_n} \le (x_n-1)(a_2+\ldots+a_{x_n}) \le (x_n-1) \, b
$$
and \eqref{eq:thm2-3} and \eqref{eq:thm2-4} leads to
$$
0<1-p_1^{b_1'-b_1''}\cdots p_k^{b_k'-b_k''}=\frac{a_2+2a_3+\ldots+(x_n-1)a_{x_n}}{n+a_2+2a_3+\ldots+(x_n-1)a_{x_n}} \le \frac{(x_n-1) \, b}{p_1^{b/2}}
$$
and
$$
0<1-p_1^{b_1'-b_1''}\cdots p_k^{b_k'-b_k''}=\frac{a_2+2a_3+\ldots+(x_n-1)a_{x_n}}{n+a_2+2a_3+\ldots+(x_n-1)a_{x_n}} \le \frac{(x_n-1) \, b}{p_j^{b_j/2}}  \quad (2 \le j \le k).
$$
By setting $\Lambda:=p_1^{b_1'-b_1''}\cdots p_k^{b_k'-b_k"}-1$, the last two inequalities show that $\Lambda \ne 0$ and
\begin{equation}\label{eq:thm2-5}
|\Lambda| \le \frac{(x_n-1)\, b}{p_1^{b/2}}.
\end{equation}
and
\begin{equation}\label{eq:thm2-6}
|\Lambda| \le \frac{(x_n-1) \, b}{p_j^{b_j/2}}  \quad (2 \le j \le k).
\end{equation}

In order to get a non-trivial lower bound for $|\Lambda|$, we apply Lemma \ref{lem:Matveev} with the parameters
$l:=k,D=1,\eta_1:=p_1,\ldots,\eta_k:=p_k$. Since $b_i=b_i'+b_i'' \ (1 \le i \le k)$ ,
we obtain that
$$
\max\{|b_1'-b_1''|,\ldots,|b_k'-b_k''|\} \le \max\{b_1,\ldots,b_k\} \le b_1+\ldots+b_k=b,
$$
and hence, we may choose $B:=b$. Further, one checks easily that $A_1:=\log{p_1},\ldots,A_k:=\log{p_k}$ is appropriate.
Hence, we get
\begin{equation} \label{eq:thm2-7}
\log|\Lambda|>-C(k,1)(1+\log{b})(\log{p_1})\cdots(\log{p_k}),
\end{equation}
where $C(k,1)=1.4\cdot30^{k+3}\cdot k^{4.5}$.
On comparing \eqref{eq:thm2-7} and \eqref{eq:thm2-5}, we arrive at
$$
b<2\cdot(1.4\cdot30^{k+3}\cdot k^{4.5})(\log{p_2})\cdots(\log{p_k})(1+\log{b})+\frac{2\log{((x_n-1) b)}}{\log{p_1}},
$$
which is exactly inequality \eqref{eq:thm2-bound}.
\end{proof}

\section{Proof of Theorem \ref{thm3}}
\begin{proof}
Let us suppose that in the equation (\ref{ELE}) we have
$$
x_n\le \frac{\log \log n \cdot \log \log \log n}{\log \log \log \log n}.
$$
By Prime number theory, $p_l=(1+o(1))l\log l$ as $l\to \infty$. Hence $\log\log p_l=(1+o(1))\log\log l$ and $$\log p_1\log p_2\dots \log p_k=e^{(1+o(1))\sum_{l=2}^k\log\log l}=e^{(1+o(1))k\log\log k}$$ as $k\to \infty$, so in Theorem \ref{thm2} the positive integer $B_0$ can be chosen as $e^{(1+o(1))k\log \log k}$. It follows from (\ref{rem2}) that
$$
\frac{(2+o(1))\log n}{\log\log\log n}\le \frac{2\log n}{\log x_n}\le B_0\le e^{(1+o(1))k\log\log k}.
$$
A routine calculation gives that $k\ge (1+o(1))\frac{\log\log n}{\log\log\log\log n}$. By Prime number theory again, we get $$x_n\ge p_k\ge (1+o(1))\frac{\log\log n\cdot \log\log\log n}{\log\log\log\log n},$$ which completes the proof of Theorem \ref{thm3}.
\end{proof}

\section{Proof of Theorem \ref{thm4}}
\begin{proof}

Let $3 \le x_n \le 10$ and consider equation \eqref{eq:main} satisfying assumptions \eqref{eq:main-ass}.
We rewrite equation \eqref{eq:main} in the form \eqref{eq:main2}, that is in the form
\begin{equation}\label{eq:main3}
n(n+a_2+2a_3+\ldots+(x_n-1)a_{x_n})=p_1^{b_1}\cdots p_k^{b_k}.
\end{equation}
We now apply Theorem \ref{thm2} to equation \eqref{eq:main3} and we obtain for each $3 \le x_n \le 10$ a Baker-type bound
$B_0$ for $b=b_1+\ldots+b_k$. We summarized the data needed for applying Theorem 2 and the bounds $B_0$ in the following table:

\vskip2cm

{\tiny
\begin{table}[!h]
\caption{}
\label{table1}
\begin{center}
\begin{tabular}{|c|c|c|c|}
\hline
$x_n$ & $k=\pi(x_n)$ & $p_i$ & $b < B_0$ \\
\hline
\hline
$3,4$ & $2$ & $p_1=2,\ p_2=3$ & $b < B_0=4.4 \cdot 10^{10}$ \\  \hline
$5,6$ & $3$ & $p_1=2,\ p_2=3, \ p_3=5$ & $b < B_0=1.6 \cdot 10^{13}$ \\  \hline
$7,8,9,10$ & $4$ & $p_1=2,\ p_2=3, \ p_3=5, \ p_4=7 $ & $b < B_0=4 \cdot 10^{15}$ \\  \hline

\end{tabular}
\end{center}
\end{table}
}

In what follows, by  using the LLL reduction method, we reduce considerably the Baker type upper bound $B_0$ for $b$ obtained
in Table \ref{table1}.

We turn our attention to inequality \eqref{eq:thm2-5}. If $|\Lambda|>\frac{1}{2}$, then $p_1=2$ and \eqref{eq:thm2-5} imply that
$$
\frac{1}{2}<\frac{(x_n-1)b}{2^{\frac{b}{2}}},
$$
which by $x_n \le 10$ leads to $b \le 16$. This shows that on assuming $b \ge 17$ we have for each $3 \le x_n \le 10$ that
$|p_1^{b_1'-b_1''}\cdots p_k^{b_k'-b_k''}-1|=|\Lambda| \le \frac{1}{2}$. By applying Lemma \ref{SmartB2}
with $z:=p_1^{b_1'-b_1''}\cdots p_k^{b_k'-b_k''}$, we obtain that
$$
|\Lambda|>\frac{|(b_1'-b_1'')\log{p_1}+\ldots+(b_k'-b_k'')\log{p_k}|}{2}.
$$
The above inequality together with \eqref{eq:thm2-5} and $p_1=2$ gives that
$$
|(b_1'-b_1'')\log{p_1}+\ldots+(b_k'-b_k'')\log{p_k}|<\frac{2(x_n-1)b}{2^{\frac{b}{2}}}.
$$
By calculus we easily find that for $b \ge 17$, we have $\frac{b}{2^{\frac{b}{2}}}<\frac{23}{2^{\frac{b}{2.1}}}$ is valid, which together with the above inequality yields
\begin{equation}\label{eq:thm3-1}
|(b_1'-b_1'')\log{p_1}+\ldots+(b_k'-b_k'')\log{p_k}|<46(x_n-1)\exp\left(-\frac{\log{2}}{2.1} \cdot b\right),
\end{equation}
where we have that
$$\max\{|b_1'-b_2''|,\ldots,|b_k'-b_k''|\} \le \max\{b_1,\ldots,b_k\} \le b_1+\ldots +b_k=b < B_0.$$
In order to reduce the size of the bound $B_0$ for $b$, we apply Lemma \ref{L_LLL_for_linearforms} to \eqref{eq:thm3-1} with $m:=k$ and we use the lattice $\mathcal L$ of $\R^m$ generated by the columns of the matrices
\begin{equation} \label{eq:thm3-A-matrix}
A:=
\begin{cases}
\begin{pmatrix}
  1 & 0 \\
  \lfloor C \cdot \log{p_1}\rfloor & \lfloor C \cdot \log{p_2}\rfloor
\end{pmatrix}
, & \mbox{if } x_n=3,4 \\
\begin{pmatrix}
  1 & 0 & 0\\
  0 & 1 & 0\\
  \lfloor C \cdot \log{p_1}\rfloor & \lfloor C \cdot \log{p_2}\rfloor & \lfloor C \cdot \log{p_3}\rfloor
\end{pmatrix}
, & \mbox{if } x_n=5,6 \\
\begin{pmatrix}
  1 & 0 & 0& 0 \\
  0 & 1 & 0& 0 \\
  0 & 0 & 1& 0 \\
  \lfloor C \cdot \log{p_1}\rfloor & \lfloor C \cdot \log{p_2}\rfloor & \lfloor C \cdot \log{p_3}\rfloor & \lfloor C \cdot \log{p_4}\rfloor
\end{pmatrix}
, & \mbox{if } x_n=7,8,9,10.
\end{cases}
\end{equation}
Further, the other relevant parameters are contained in the following table:

{\tiny
\begin{table}[!h]
\caption{}
\label{table2}
\begin{center}
\begin{tabular}{|c|c|c|c|c|c|c|c|c|c|c|c|}
\hline
$x_n$ & $k=\pi(x_n)$ & $m$ & $x_i$ & $\alpha_i$ & $X_i$ & $X_0$ & $q$ & $H$ & $C_2$ & $C_3$ & $C$ \\
\hline
\hline
$3,4$ & $2$ & $2$ & $b_i'-b_i''$ & $\log{p_i}$ & $B_0=4.4 \cdot 10^{10}$ & $B_0=4.4 \cdot 10^{10}$ & $1$ & $b$ & $46\cdot(4-1)=138$ & $\frac{\log{2}}{2.1}$ & $10^{23}$\\  \hline
$5,6$ & $3$ & $3$ & $b_i'-b_i''$ & $\log{p_i}$ & $B_0=1.6 \cdot 10^{13}$ & $B_0=1.6 \cdot 10^{13}$ & $1$ & $b$ & $46\cdot(6-1)=230$ & $\frac{\log{2}}{2.1}$ & $10^{42}$\\  \hline
$7,8,9,10$ & $4$ & $4$ & $b_i'-b_i''$ & $\log{p_i}$ & $B_0=4 \cdot 10^{15}$ & $B_0=4 \cdot 10^{15}$ & $1$ & $b$ & $46\cdot(10-1)=414$ & $\frac{\log{2}}{2.1}$ & $10^{67}$\\  \hline
\end{tabular}
\end{center}
\end{table}
}

By using the LLL algorithm implemented in the program package MAGMA, we obtained an LLL-reduced basis of the lattice $\mathcal L$ and by Lemma \ref{L_shortest_vector_by_LLL} we see that $C_1:=\sqrt{||b_1||^2/2^{m-1}}$ is a lower bound for the length of the shortest nonzero vector in the lattice. By using the program package MAGMA again, it turned out that in our cases the assumption
$$C_1^2>T^2+S$$
is fulfilled and hence Lemma \ref{L_LLL_for_linearforms} shows that $b \le B_{0,new}$, where
\begin{equation}\label{eq:thm3-2}
B_{0,new}=
\begin{cases}
97, & \mbox{if } x_n=3,4 \\
219, & \mbox{if } x_n=5,6 \\
372, & \mbox{if } x_n=7,8,9,10.
\end{cases}
\end{equation}

\noindent We iterated the above process two more times according to the following table:

\vskip1cm

{\tiny
\begin{table}[!h]
\caption{}
\label{table3}
\begin{center}
\begin{tabular}{|c|c|c|c|c|}
\hline
$x_n$ & $X_i$ & $X_0$ & $C$ & $B_{0,new}$ \\
\hline
\hline
$3,4$ & $97$ & $97$ & $10^{5}$ & $37$\\  \hline
$3,4$ & $37$ & $37$ & $10^{4}$ & ${\bf 33}$\\  \hline
$5,6$ & $219$ & $219$ & $10^{10}$ & $66$\\  \hline
$5,6$ & $66$ & $66$ & $5 \cdot 10^{8}$ & ${\bf 60}$\\  \hline
$7,8,9,10$ & $372$ & $372$ & $10^{14}$ & $101$\\  \hline
$7,8,9,10$ & $101$ & $101$ & $2 \cdot 10^{12}$ & ${\bf 87}$\\  \hline
\end{tabular}
\end{center}
\end{table}
}

\noindent From Table \ref{table3} we see that
\begin{equation}\label{eq:thm3-3}
b \le B_{0,red}:=
\begin{cases}
33, & \mbox{if } x_n=3,4 \\
60, & \mbox{if } x_n=5,6 \\
87, & \mbox{if } x_n=7,8,9,10.
\end{cases}
\end{equation}
Since $b=b_1+\ldots+b_k$, \eqref{eq:thm3-3} implies that $\max\{b_1,b_2,\ldots,b_k\} \le b \le B_{0,red}$.
However, using inequalities \eqref{eq:thm2-6} and the same approach as above, we can sharpen upon the upper bound $B_{0,red}$
for $b_2,\ldots,b_k$. The reason is that for $2,3,\ldots, k$ the primes satisfy $2<p_2<p_3\ldots<p_k$, that is, each prime is strictly
greater than $2$. Since $2<p_2<p_3\ldots<p_k$, we see that for $b \ge 17$ we may assume that in \eqref{eq:thm2-6} one has $|\Lambda|<\frac{1}{2}$. Hence the combination of Lemma \ref{SmartB2} and \eqref{eq:thm3-3} yields
\begin{equation} \label{eq:thm3-4}
|(b_1'-b_1'')\log{p_1}+\ldots+(b_k'-b_k'')\log{p_k}|<2(x_n-1)B_{0,red}\exp\left(-\frac{\log{p_j}}{2} \cdot b\right).
\end{equation}
We apply Lemma \ref{L_LLL_for_linearforms} to inequality \eqref{eq:thm3-4} with $m:=k$ and
we use the lattice $\mathcal L$ of $\R^m$ generated by the columns of the matrices $A$ given by \eqref{eq:thm3-A-matrix}.
Further, we set the parameter $C_3$ as $C_3:=\frac{\log{p_j}}{2}, \ (2 \le j \le k)$,
and according to \eqref{eq:thm3-3} and \eqref{eq:thm3-4}, we may choose
\begin{equation}
C_2:=
\begin{cases}
198=2 \cdot (4-1) \cdot 33, & \mbox{if } x_n=3,4 \\
600=2 \cdot(6-1) \cdot 60, & \mbox{if } x_n=5,6 \\
1566=2 \cdot (10-1) \cdot 87, & \mbox{if } x_n=7,8,9,10.
\end{cases}
\end{equation}

\noindent We iterated the application of Lemma 3 two times for each $b_i \ (2 \le i \le k)$, and the other relevant parameters are given in the following table:

\vskip2cm

{\tiny
\begin{table}[!h]
\caption{}
\label{table4}
\begin{center}
\begin{tabular}{|c|c|c|c|c|c|}
\hline
$x_n$ & $k=\pi(x_n)$ & $X_i$ & $X_0$ & $C$ & $\textrm{Bounds for} \ b_i \ (2 \le i \le k)$ \\
\hline
\hline
$3,4$       &2 & $33$ & $33$ & $10^{4}$ & $b_2 \le 20$\\  \hline
$3,4$       &2 & $20$ & $20$ & $10^{4}$ & $b_2 \le {\bf 18}$\\  \hline \hline
$5,6$       &3 & $60$ & $60$ & $5 \cdot 10^{7}$ & $b_2 \le  35$\\  \hline
$5,6$       &3 & $35$ & $35$ & $5 \cdot 10^{6}$ & $b_2 \le {\bf 32}$\\  \hline
$5,6$       &3 & $60$ & $60$ & $5 \cdot 10^{7}$ & $b_3 \le 24$\\  \hline
$5,6$       &3 & $24$ & $24$ & $5 \cdot 10^{6}$ & $b_3 \le {\bf 21}$\\  \hline\hline
$7,8,9,10$  &4 & $87$ & $87$ & $2 \cdot 10^{12}$ & $b_2 \le 54$\\  \hline
$7,8,9,10$  &4 & $54$ & $54$ & $2 \cdot 10^{10}$ & $b_2 \le {\bf 49}$\\  \hline
$7,8,9,10$  &4 & $87$ & $87$ & $2 \cdot 10^{12}$ & $b_3 \le 37$\\  \hline
$7,8,9,10$  &4 & $37$ & $37$ & $7 \cdot 10^{9}$ & $b_3 \le {\bf 31}$\\  \hline
$7,8,9,10$  &4 & $87$ & $87$ & $2 \cdot 10^{12}$ & $b_4 \le 30$\\  \hline
$7,8,9,10$  &4 & $54$ & $54$ & $5 \cdot 10^{9}$ & $b_4 \le {\bf 25}$\\  \hline
\end{tabular}
\end{center}
\end{table}
}

\noindent Now, the combination of \eqref{eq:thm3-3} and Table \ref{table4} gives sharp upper bounds for the quantities $b_i \ (1 \le i \le k)$.
Namely, we obtain
\begin{equation}\label{eq:thm3-5}
\begin{cases}
b_1 \le b \le 33, b_2 \le 18, & \mbox{if } x_n=3,4 \\
b_1 \le b \le 60, b_2 \le 32, b_3 \le 21 & \mbox{if } x_n=5,6 \\
b_1 \le b \le 87, b_2 \le 49, b_3 \le 31, b_4 \le 25   , & \mbox{if } x_n=7,8,9,10.
\end{cases}
\end{equation}
From \eqref{eq:thm3-5} and \eqref{eq:thm2-1} it is clear that all the unknowns $a_2,a_3,\ldots,a_{x_n}$ in equation \eqref{eq:main}
(and in \eqref{eq:main2}) are also bounded. Observe that since $2^{a_2}3^{a_3}\cdots x_n^{a_{x_n}}=p_1^{b_1}\cdots p_k^{b_k}$ equation \eqref{eq:main} can be rewritten as
\begin{equation} \label{eq:thm3-6}
n^2+(a_2+2a_3+\ldots+(x_n-1)a_{x_n})n-p_1^{b_1}\cdots p_k^{b_k}=0,
\end{equation}
where the unknowns $a_i,b_i$ are already bounded. Further, for given $a_i,b_i$ equation \eqref{eq:thm3-6}
is a quadratic equation in $n$, which can have a solution only if its discriminant is a square.
We implemented for each $3 \le x_n \le 10$ this observation in the program package MAGMA and we obtained all the solutions
of equation \eqref{eq:main}. For the sake of completeness, we illustrate our enumeration method in the case, where $x_n=10$.
Note, that the cases $3 \le x_n \le 9$ can be done in the same way and that the case $x_n=10$ was the most time consuming
case from computational point of view.

\medskip

\noindent Let $x_n=10$. Then, clearly $k=4, a_{10} \ge 1$ and $b=b_1+b_2+b_3+b_4$.
According to \eqref{eq:thm3-5} and \eqref{eq:thm2-1}, we then have
\begin{equation}\label{eq:thm3-7}
\begin{cases}
b_1=a_2+2a_4+a_6+3a_8+a_{10} \le B_1:=87, \\
b_2=a_3+a_6+2a_9 \le B_2:=49, \\
b_3=a_5+a_{10} \le B_3:=31, \\
b_4=a_7 \le B_4:=25, \\
b=a_2+a_3+2a_4+a_5+2a_6+a_7+3a_8+2a_9+2a_{10} \le B_1:=87.\\
\end{cases}
\end{equation}

\noindent Based on \eqref{eq:thm3-7}, we performed a simple search implemented in the program package MAGMA:
\vskip.2cm

{\tt
\vskip.1cm
for $a_{10}:=1$ to $B_3$ do
\vskip.1cm
for $a_7:=0$ to $B_4$ do
\vskip.1cm
for $a_5:=0$ to $\min(B_1-a_7-2a_{10},B_3-a_{10})$ do
\vskip.1cm
for $a_9:=0$ to $\min\left(\lfloor B_2/2 \rfloor, \lfloor (B_1-a_7-2a_{10}-a_5)/2 \rfloor\right)$ do
\vskip.1cm
for $a_6:=0$ to $\min\left(B_2-2a_9, B_1-a_7-2a_{10}-a_5-2a_9 \right)$ do
\vskip.1cm
for $a_3:=0$ to $\min\left(B_2-2a_9-a_6, B_1-a_7-2a_{10}-a_5-2a_9-2a_6 \right)$ do
\vskip.1cm
for $a_8:=0$ to $\lfloor (B_1-a_7-2a_{10}-a_5-2a_9-2a_6-a_3)/3\rfloor$ do
\vskip.1cm
for $a_4:=0$ to $\lfloor(B_1-a_7-2a_{10}-a_5-2a_9-2a_6-a_3-3a_8)/2 \rfloor$ do
\vskip.1cm
for $a_2:=0$ to $B_1-a_7-2a_{10}-a_5-2a_9-2a_6-a_3-3a_8-2a_4$ do
\vskip.1cm
$s:=a_2+2a_3+3a_4+4a_5+5a_6+6a_7+7a_8+8a_9+9a_{10}$;
\vskip.1cm
$discr:=s^2+4 \cdot 2^{a_2+2a_4+a_6+3a_8+a_{10}}3^{a_3+a_6+2a_9}5^{a_5+a_{10}}7^{a_7}$;
\vskip.1cm
if IsSquare($discr$)=true then
\vskip.1cm
$n:=(-s+\sqrt{discr})/2$;
\vskip.1cm
$a_1:=n-(a_2+a_3+a_4+a_5+a_6+a_7+a_8+a_9+a_{10})$;
\vskip.1cm
if $a_1 \ge 0$ then
\vskip.1cm
Put the values $[n,a_1,a_2,a_3,a_4,a_5,a_6,a_7,a_8,a_9,a_{10}]$ into a list;
\vskip.1cm
end if;
\vskip.1cm
end if;

end for;
}

\vskip.2cm\noindent
The computation time of the above program was $50530$ seconds (about $14.05$ hours) on a laptop with a processor Intel Core i7, 11800H with 16GB of RAM.
As mentioned before, this was the most time consuming part of our computation. The total time needed for $3 \le x_n \le 10$
was $61500$ seconds (about $17.1$ hours).

\end{proof}


\end{document}